\documentclass[12pt, a4paper]{article}
\usepackage{geometry}
\usepackage{mathptm,enumerate,xcolor,fullpage}
\usepackage{amsmath}
\usepackage{dsfont}
\usepackage{amsthm}
\usepackage{graphicx}
\usepackage{mathrsfs}
\usepackage{cite}
\usepackage{newtxtext}
\usepackage{lipsum}
\usepackage{multirow} 
\usepackage{amssymb}
\usepackage{geometry}
\usepackage{mathtools}
\usepackage{hyperref}
\usepackage{mathrsfs}
\usepackage{upgreek}
\usepackage[capitalize,nameinlink]{cleveref}

\theoremstyle{definition}
\newtheorem{theorem}{Theorem}[section]
   
\newtheorem{lemma}[theorem]{Lemma}

\theoremstyle{definition}

\newcommand{\E}{\mathbf{E}}
\renewcommand{\P}{\mathbf{P}}
\newcommand{\Prob}[1]{\P\left\{#1\right\}}

\newcommand{\N}{\mathbb{N}}

\newcommand{\R}{\mathbb{R}}

\newcommand{\sF}{\mathcal{F}}

\newcommand{\sN}{\mathcal{N}}
\newcommand{\sZ}{\mathcal{Z}}

\newcommand{\one}{\mathbf{1}}

\newcommand{\eps}{\varepsilon}

\renewcommand{\epsilon}{\varepsilon}

\newcommand{\nt}{\lfloor nt \rfloor}
\newcommand{\ns}{\lfloor ns \rfloor}

\DeclareMathOperator{\cov}{Cov}
\DeclareMathOperator{\var}{Var}

\renewenvironment{abstract}{%
  \noindent\bfseries\abstractname:\normalfont}{}

\newlength{\querylen}
\usepackage{fancybox}

\title{Convergence of Random Walks in $\ell_p$-Spaces of Growing
  Dimension}

\author{Bochen Jin}

\date{\today}

\begin{document}
\maketitle

\begin{abstract}
  We prove the limit theorem for paths of random walks with $n$ steps in
  $\R^d$ as $n$ and $d$ both go to infinity. For this, the paths 
  are viewed as finite metric spaces equipped with the
  $\ell_p$-metric for $p\in[1,\infty)$. Under the assumptions that all
  components of each step are uncorrelated, centered, have finite $2p$-th 
  moments, and are identically distributed, we show that such random
  metric space converges in probability to a deterministic limit space
  with respect to the Gromov-Hausdorff distance. This result
  generalises earlier work by Kabluchko and Marynych \cite{MR4828862}
  for $p=2$.


\end{abstract}

\section{Introduction}
Consider a $d$-dimensional random walk defined by
\begin{displaymath}
S_0^{(d)} = 0, \quad S_n^{(d)} = X_1^{(d)} + X_2^{(d)} + \cdots +
X_n^{(d)}, \quad n \in \N,
\end{displaymath}
where $X_i^{(d)} = (X_{i,1}^{(d)}, \ldots, X_{i,d}^{(d)})$, $i \geq
1$, are independent identically distributed random vectors in
$\R^d$, and denote $S_i^{(d)} = (S_{i,1}^{(d)}, \ldots, S_{i,d}^{(d)})$.

Let $\ell_2$ be the space of square-summable real sequences with norm
denoted by $\|\cdot\|_2$. The values of this random walk 
\begin{displaymath}
  \sZ_n^{(d)}=\{S_0^{(d)},\ldots,S_n^{(d)}\},\quad n\in\N.
\end{displaymath}
are considered as a finite metric space which is embedded in $\R^d$
with the induced Euclidean metric.

In the regime when the dimension $d$ is fixed, provided
that $\E X_1^{(d)} = 0$ and $\E \|X_1^{(d)}\|_2^2 =1$, Donsker's
invariance principle implies that, after rescaling
by $n^{-1/2}$, the random set $\sZ_n^{(d)}$ converges in distribution
to the path of a $d$-dimensional Brownian motion on $[0,1]$.

When both $n$ and $d$ tend to
infinity, under the square integrability and several further
assumptions listed in \cite{MR4828862}, the
random metric space $(n^{-1/2}\sZ_n^{(d)}, \|\cdot\|_2)$
converges in probability to the Wiener spiral with respect to the
Gromov-Hausdorff distance. The latter space is the space of indicator
functions $\one_{[0,t]}$, $t\in[0,1]$, embedded in $L^2([0,1])$,
which is isometric to the interval $[0,1]$ equipped with the metric
$r(t,s) = \sqrt{|t - s|}$.

The \textit{Gromov--Hausdorff distance} between metric
spaces $\mathbb{X}=(X, \rho_X)$ and $\mathbb{Y}=(Y,
  \rho_Y)$ is defined as
\begin{displaymath}
  d_{GH}(\mathbb{X}, \mathbb{Y})=\inf_{i: X \hookrightarrow Z, j: Y \hookrightarrow Z} d_H
  \big(i(X), j(Y) \big),
\end{displaymath}
where the infimum is taken over all isometric embeddings $i$ and $j$
into all possible metric spaces $(Z, d)$ which can embed $X \
\text{and} \ Y$. The Hausdorff distance between sets $F$ and $H$
in $(Z,d)$ is defined as
\begin{displaymath}
  d_H(F, H)=\inf \{ \epsilon > 0: F \subset H^{\epsilon} \
  \text{and}\ H \subset F^{\epsilon} \},
\end{displaymath}
where $F^{\epsilon}=\{x: d(x, F) < \epsilon \}$ is the 
$\epsilon$-neighbourhood of $F$, see \cite[Chapter 7]{MR1835418}. 

We replace the $l_2$-metric on the space of sequences with the
$\ell_p$-metric for a general $p\in[1,\infty)$. The studies of random
metric spaces rely on identifying the spaces up to isometries.
Contrary to the $\ell_2$ setting, which admits a large group of
rotations as isometries, the isometry group of $\ell_p$ for $p \neq 2$
is far more constrained, including the permutations of components, see 
\cite{MR105017} and \cite[Theorem~7.4.1]{MR4249569}. Furthermore, 
while we have the identity $\|x + y\|_2^2 = \|x\|_2^2 + \|y\|_2^2 + 2\langle x, y \rangle,$
no analogous simple expression exists for $\|x + y\|_p^p$ with $p
\neq 2$. This complicates the analysis of path of the random walk.

It should be noted that Kabluchko and Marynych \cite{MR4828862}
established that for subsets of $\ell_2$, convergence in the
Gromov–Hausdorff sense is equivalent to convergence in the Hausdorff
distance up to isometries of $\ell_2$, that is distance for the
two subsets defined by taking the infimum over
the Hausdorff distance between their images under all possible isometries of
$\ell_2$.
However, this equivalence fails for compact subsets of $\ell_p$ when
$p \neq 2$. For instance, the two-point metric spaces
$F=\{(0,0,\ldots),(1,0,0,\ldots)\}$ and
$H=\{(0,0,\ldots),(a^{1/p},(1-a)^{1/p},0,\ldots)\}$ for any
$a\in(0,1)$ with the $\ell_p$-metric are isometric for any
$p\in[1,\infty)$ and so the Gromov--Hausdorff distance between them
vanishes, while it is not possible to map $F$ to $H$ using an isometry
of $\ell_p$ if $p\neq2$.

Fix a $p\in[1,\infty)$. Impose a special structure on the increments
of the random walk. Namely, we assume that
\begin{equation}
  \label{equ_spe_form}
  X_1^{(d)} = d^{-1/p}(\xi_{1}, \ldots, \xi_{d}),
\end{equation}
where $\xi_{1}, \ldots, \xi_{d}$ are uncorrelated random
variables which share the same distribution with a centered and
$2p$-integrable nontrivial random variable $\xi$. Denote $\E
\xi^2=\sigma^2$. Let $M_p$ denote the $p$-th absolute moment of the
standard normal distribution. 
\begin{theorem} 
\label{the_pair_lp}
  Let $p \in [1 ,\infty)$ and $d = d(n)$ be an arbitrary
  sequence of positive integers such that $d(n) \to \infty$
  as $n\to\infty$. Consider a random walk with increments given by 
  \eqref{equ_spe_form}. Then, as $n \to \infty$, the random metric space 
  $(n^{-1/2}\sZ_n^{(d)},\|\cdot\|_p)$, 
  converges in probability to $\big([0,1], \sqrt{|t-s|} \sigma
  M_p^{1/p} \big)$ under the Gromov-Hausdorff distance.
\end{theorem}

The paper is organised as follows. In Section~\ref{sec:3},
we provide the univariate and bivariate moment convergence theorems,
which serve as the main tools for proving the limit
theorem. Section~\ref{sec:uc} contains some auxiliary theorems and the
proof of the main result.

\section{Moment convergence theorem}
\label{sec:3}
We need the following results.
\begin{theorem}[Moment convergence theorem]
  \label{thm:mom_gene}
  Let $\eta, \eta_1, \eta_2, \ldots$ be independent identically
  distributed random variables with $\E\eta=\mu$ and
  $\var \eta = \sigma^2$, and $S_n = \eta_1+\ldots+\eta_n$,
  $n \geq 1$. Then
  \begin{displaymath}
    \E \bigg|\frac {S_n - n\mu}{\sigma \sqrt{n}} \bigg|^p
    \to M_p
    = 2^{p/2} \frac{1}{\sqrt{\pi}} \Gamma \bigg( \frac{p + 1}{2}
    \bigg),
  \end{displaymath}
  if $p \in (0,2)$ and for $p \geq 2$ if $\E |\eta|^p < \infty$.
\end{theorem}

\begin{proof}
  If $p \in (0,2)$, then
  \begin{displaymath}
      \sup_n \E \Biggl(\bigg|\frac{S_n-n\mu}{\sigma \sqrt{n}} \bigg|^p\Biggl)^r
      \leq \Biggl(\sup_n \E \bigg(\frac{S_n-n\mu}{\sigma \sqrt{n}} \bigg)^2\Bigg)^{\frac{pr}{2}}
      = 1 < \infty
  \end{displaymath}
  by choosing $r > 1$ and $pr < 2$, which implies that
  \begin{displaymath}
    \bigg\{\bigg|\frac{S_n-n\mu}{\sigma \sqrt{n}} \bigg|^p, \ n\geq 1\bigg\}
    \quad\text{is uniformly integrable}.
  \end{displaymath}
  Furthermore, Theorem
  \ref{thm:mom_gene} holds by the central limit theorem and the
  continuous mapping theorem, that is,
  \begin{displaymath}
    \bigg|\frac{S_n-n\mu}{\sigma \sqrt{n}} \bigg|^p
    \overset{d} \to |\sN(0,1)|^p\quad\text{as}\ n\to\infty.
  \end{displaymath}
  The case of $p \geq 2$ is proved in Theorem 7.5.1 from
  \cite{MR2977961}.
\end{proof}

\begin{lemma}[Marcinkiewicz–Zygmund inequality. 
See Corollary 3.8.2 in \cite{MR2977961}]
\label{lem_Mar}
Let $p \geq 1$. Suppose that $X, X_1, \ldots, X_n$ are
independent, identically distributed random variables with
mean $0$ and $\E |X|^p < \infty$.
Set $S_n = \sum_{k = 1}^{n} X_k$, $n \geq 1$.
Then there exists a constant $B_p$ depending only on $p$ such that
\begin{displaymath}
  \E |n^{-1/2}S_n|^p \leq
  \begin{cases}
    B_p n^{1-p/2} \E |X|^p, & \text{when $1 \leq p \leq 2$},\\
    B_p \E |X|^p, & \text{when $p \geq 2$}.
  \end{cases}
\end{displaymath}
\end{lemma}

\begin{theorem}[Bivariate moment convergence theorem]
  \label{them_mom_two_dimen}
  Let $(X_1, Y_1),\ldots,(X_n, Y_n)$ be independent copies of a
  centered $2p$-integrable random vector $(X, Y)$ with the covariance
  matrix $\Sigma$. Denote $S_n = X_1+\cdots+X_n$
  and $Z_n = Y_1+\cdots+Y_n$. Then
  \begin{equation}
    \label{equ_mom_two_dimen}
    \E \big| n^{-1/2}S_n \big|^p
    \big| n^{-1/2}Z_n\big|^p \to \E |\eta_1 \eta_2|^p
    \quad\text{as}\ n\to\infty,
  \end{equation}
  where $(\eta_1, \eta_2) \sim\mathcal{N} (0, \Sigma)$.
\end{theorem}

The proof relies on the following lemma.

\begin{lemma}[The $c_r$ inequality. See Theorem 2.2 in \cite{MR2977961}]
  \label{lem_cr}
  Let $r > 0$. Suppose that $\E |X|^r < \infty$ and
  $\E |Y|^r<\infty$. Then
  \begin{displaymath}
    \E |X + Y|^r \leq c_r (\E |X|^r + \E |Y|^r),
  \end{displaymath}
  where $c_r = 1$ when $r \leq 1$ and $c_r = 2^{r - 1}$ when $r \geq 1$.
\end{lemma}


\begin{proof}[Proof of Theorem \ref{them_mom_two_dimen}.]
  It suffices to show that
  \begin{displaymath}
    \big\{ | n^{-1/2}S_n |^p |
    n^{-1/2}Z_n |^p, n \geq 1 \big\}
  \end{displaymath}
  is uniformly integrable. Fix an $\eps > 0$. Choose $A$ and $A'$ 
  large enough to ensure that
  \begin{displaymath}
    \E |X|^{2p} \mathbf{1}_{|X| > A} < \eps,\quad\text{and}
    \quad \E |Y|^{2p} \mathbf{1}_{|Y| > A'} < \eps.
  \end{displaymath}

  Set 
  \begin{align*}
    X_k^{\prime}
    & = X_k \mathbf{1}_{|X_k| \leq A} - \E (X_k\mathbf{1}_{|X_k| \leq A}), \\
    X_k^{\prime \prime}
    & = X_k \mathbf{1}_{|X_k| > A} - \E (X_k \mathbf{1}_{|X_k| > A})
  \end{align*}
  and
  \begin{displaymath}
    S_n^{\prime} = \sum_{k = 1}^{n} X_k^{\prime}, \qquad
    S_n^{\prime \prime} = \sum_{k = 1}^{n} X_k^{\prime \prime}
  \end{displaymath}
  Similarly, define
  $Y_k^{\prime}$, $Y_k^{\prime \prime}$ and
  $Z_n^{\prime}$, $Z_n^{\prime \prime}$.
  Note that $\E X_k^{\prime} = \E X_k^{\prime \prime} = 0$,
  $X_k^{\prime} + X_k^{\prime \prime} =
  X_k$, $S_n^{\prime} + S_n^{\prime \prime} = S_n$
  and $\E Y_k^{\prime} = \E Y_k^{\prime \prime} = 0$,
  $Y_k^{\prime} + Y_k^{\prime \prime} =
  Y_k$, $Z_n^{\prime} + Z_n^{\prime \prime} = Z_n$.

  Let $a > 0$. Note that
  \begin{align*}
    \E | n^{-1/2}S_n^{\prime} |^p 
    | n^{-1/2}Z_n^{\prime} |^p
    &\mathbf{1}_{|n^{-1/2}S_{n}^{\prime}| |n^{-1/2}Z_{n}^{\prime}| > a}\\
    &\leq \frac{1}{a^p} \E
      | n^{-1/2}S_n^{\prime} |^{2p}
      | n^{-1/2}Z_n^{\prime} |^{2p}
      \mathbf{1}_{|n^{-1/2}S_{n}^{\prime}| |n^{-1/2}Z_{n}^{\prime}| > a} \\
    & \leq \frac{1}{a^p} \E
      | n^{-1/2}S_n^{\prime} |^{2p}
      | n^{-1/2}Z_n^{\prime} |^{2p} \\
    & \leq \frac{1}{a^p} \big(\E
      |n^{-1/2}S_{n}^{\prime}|^{4p}
      \E |n^{-1/2}Z_{n}^{\prime}|^{4p}\big)^{1/2} \\
    & = \frac{1}{a^p} B_{4p} \big(\E |X_1^{\prime}|^{4p}
      \E |Y_1^{\prime}|^{4p}\big)^{1/2} \\
    & \leq \frac{1}{a^p} B_{4p} \big((2A)^{4p} (2A')^{4p}\big)^{1/2},
  \end{align*}
  where the penultimate step follows from Lemma \ref{lem_Mar}. Also
  \begin{align*}
    \E |n^{-1/2}S_n^{\prime\prime}|^{p}
    |n^{-1/2}Z_n^{\prime\prime}|^{p}
    &\mathbf{1}_{|n^{-1/2}S_n^{\prime\prime}||n^{-1/2}Z_n^{\prime\prime}| > a}
      \leq \E|n^{-1/2}S_n^{\prime\prime}|^{p}
      |n^{-1/2}Z_n^{\prime\prime}|^{p}  \\
    & \leq \big(\E |n^{-1/2}S_n^{\prime\prime}|^{2p}
      \E |n^{-1/2}Z_n^{\prime\prime}|^{2p}\big)^{1/2} \\
    & \leq \big( B_{2p} B_{2p} \E
      |X_1^{\prime\prime}|^{2p}
      \E |Y_1^{\prime\prime}|^{2p}\big)^{1/2} 
      \leq B_{2p} 2^{2p} \eps,
  \end{align*}
  where the last inequality follows from Lemma \ref{lem_cr}, that is,
  \begin{align*}
    \E |X_1^{\prime \prime}|^{2p}
    & = \E \big|X_1 \mathbf{1}_{|X_1| > A}
      - \E X_1 \mathbf{1}_{|X_1| > A} \big|^{2p} \\
    & \leq 2^{2p - 1} \big( \E |X_1 \mathbf{1}_{|X_1| > A}|^{2p}
      + \E |X_1 \mathbf{1}_{|X_1| > A}|^{2p} \big)
      = 2^{2p} \E |X_1|^{2p} \mathbf{1}_{|X_1| > A} \leq 2^{2p}\eps.
  \end{align*}

  Similarly, 
  \begin{align}
    \label{equ_mom_3}
    \begin{split}
      \E | n^{-1/2}S_n^{\prime} |^p
      |n^{-1/2}Z_n^{\prime\prime}|^{p}
      &\mathbf{1}_{|n^{-1/2}S_{n}^{\prime}| |n^{-1/2}Z_n^{\prime\prime}| > a} 
        \leq \E|n^{-1/2}S_n^{\prime}|
        |n^{-1/2}Z_n^{\prime}|\\
      &\leq \big(\E | n^{-1/2}S_n^{\prime} |^{2p}
        \E |n^{-1/2}Z_n^{\prime\prime}|^{2p}\big)^{1/2} \\
      & \leq \big(B_{2p} B_{2p}
        \E |X_1^{\prime}|^{2p}
        \E |Y_1^{\prime\prime}|^{2p}\big)^{1/2} \\
      & = B_{2p} \big(\E |X_1^{\prime}|^{2p} \E |Y_1^{\prime\prime}|^{2p}\big)^{1/2} 
      \leq B_{2p} \big(2^{2p} (2A)^{2p} \eps\big)^{1/2},
    \end{split}
  \end{align}
  and 
  \begin{equation}
    \label{equ_mom_4}
    \E |n^{-1/2}S_n^{\prime\prime}|^{p}
    | n^{-1/2}Z_n^{\prime} |^p
    \mathbf{1}_{|n^{-1/2}S_n^{\prime\prime}| |n^{-1/2}Z_{n}^{\prime}| > a} 
    \leq B_{2p} \big(2^{2p} (2A')^{2p} \eps\big)^{1/2}.
  \end{equation}

  Hence, we conclude that
  \begin{align*}
    & \Big\{ | n^{-1/2}S_n^{\prime} |^p
      | n^{-1/2}Z_n^{\prime} |^p, n \geq 1 \Big\},
      \qquad \Big\{ | n^{-1/2}S_n^{\prime} |^p
      |n^{-1/2}Z_n^{\prime\prime}|^{p}, n \geq 1\Big\}, \\
    & \Big\{ |n^{-1/2}S_n^{\prime\prime}|^{p}
      | n^{-1/2}Z_n^{\prime} |^p, n \geq 1 \Big\},
      \qquad \Big\{ |n^{-1/2}S_n^{\prime\prime}|^{p}
      |n^{-1/2}Z_n^{\prime\prime}|^{p}, n \geq 1 \Big\},
  \end{align*}
  are uniformly integrable. By Lemma \ref{lem_cr},
  \begin{align*}
    |n^{-1/2}S_n|^p|n^{-1/2}Z_n|^p
    &=|n^{-1/2}(S_n'+S_n'')|^p|n^{-1/2}(Z_n'+Z_n'')|^p\\
    &\leq 2^{2p-2}\big(|n^{-1/2}S_n'|^p+|n^{-1/2}S_n''|^p\big)
      \big(|n^{-1/2}Z_n'|^p+|n^{-1/2}Z_n''|^p\big).
  \end{align*}
  Finally, the proof is completed by the fact that the sum of uniformly 
  integrable sequences is again uniformly integrable.
\end{proof}  

\section{Convergence of the $\ell_p$-metric of random walks}
\label{sec:uc}
The proof of Theorem \ref{the_pair_lp} relies on the following theorems,
while they follow the general scheme of \cite{MR4828862}, substantial 
adjustments are necessary to handle the $\ell_p$-case with $p\neq 2$.
\begin{theorem}
\label{thm_general_l_p}
Let $p \in [1, \infty)$. Consider a random walk with increments given by
\eqref{equ_spe_form}. Then
\begin{displaymath}
  n^{-p/2}\|S_{\nt}^{(d)}\|_p^p
  \overset{p} \to t^{p/2}\sigma^p M_p\quad\text{as}\ n\to\infty
\end{displaymath}
for all $t\in[0,1]$.
\end{theorem}

\begin{proof}
  Without loss of generality, let $t=1$.
  By the definition of convergence in probability, we need to
  verify that
  \begin{equation}
    \label{equ_aim_lp}
    \Prob{\bigg|n^{-p/2}\sum_{i =1}^{d}|S_{n,i}^{(d)}|^p
      - \sigma^p M_p\bigg| > \eps}
    \to 0 \qquad \text{as} \ n\to\infty.
  \end{equation}
  Markov's inequality implies that \eqref{equ_aim_lp} is bounded above
  by
  \begin{align*}
    &\eps^{-2}\E
      \bigg(n^{-p/2}\sum_{i=1}^{d}|S_{n,i}^{(d)}|^p
      -\sigma^pM_p\bigg)^2\\
    &\hspace{0.5cm}=\eps^{-2}\Bigg(\E\bigg(n^{-p/2}\sum_{i=1}^{d}|S_{n,i}^{(d)}|^p\bigg)^2
      +\sigma^{2p}M_p^2-2\sigma^pM_p
      \E\bigg(n^{-p/2}\sum_{i=1}^{d}|S_{n,i}^{(d)}|^p\bigg)\Bigg)\\
    &\hspace{0.5cm}=\eps^{-2}\Bigg(n^{-p}d
      \E |S_{n,1}^{(d)}|^{2p}
      +n^{-p}\sum_{1\leq i\neq j\leq d}\E
       |S_{n,i}^{(d)}|^{p}|S_{n,j}^{(d)}|^{p}
      +\sigma^{2p}M_p^2-2 n^{-p/2}\sigma^p
      M_pd\E |S_{n,1}^{(d)}|^{p}\Bigg)\\
    &\hspace{0.5cm}=\eps^{-2}(A_1+A_2+A_3),
  \end{align*}
  where
  \begin{displaymath}
    A_1=n^{-p}d\E|S_{n,1}^{(d)}|^{2p},
  \end{displaymath}
  \begin{displaymath}
    A_2=n^{-p}\sum_{1\leq i\neq j\leq d}\cov\big(
    |S_{n,i}^{(d)}|^{p},
    |S_{n,j}^{(d)}|^{p}\big),
  \end{displaymath}
  and
  \begin{displaymath}
    A_3=n^{-p}d(d-1)\big(\E
    |S_{n,1}^{(d)}|^{p}\big)^2
    +\sigma^{2p}M_p^2-2n^{-p/2} \sigma^p
    M_pd\E |S_{n,1}^{(d)}|^{p}.
  \end{displaymath}

  Let $(\xi_{1}^{(k)},\ldots,\xi_{d}^{(k)})$, $1\leq k\leq n$, be independent 
  copies of $(\xi_1,\ldots,\xi_d)$.
  Denote
  \begin{displaymath}
    b_{nij,p}=\cov\Big(n^{-p/2}
    \big|\xi_{i}^{(1)}+\xi_{i}^{(2)}+\cdots+\xi_{i}^{(n)}\big|^p,n^{-p/2}
    \big|\xi_{j}^{(1)}+\xi_{j}^{(2)}+\cdots+\xi_{j}^{(n)}\big|^p\Big)
  \end{displaymath}
  for all $1\leq i\neq j\leq d$.
  Thus, by Lemma \ref{lem_Mar}, 
  \begin{displaymath}
    A_1=d^{-1}\E n^{-p}|\xi_{1}+\xi_{1}^{(2)}+\cdots+\xi_{1}^{(n)}|^{2p}\leq
    d^{-1}B_{2p}\E|\xi|^{2p},
  \end{displaymath}
  where $B_{2p}$ is a constant depending
  only on $p$, and the term $A_1$ converges to 0 as $d\to\infty$.
  
  By Theorem \ref{them_mom_two_dimen}, the term $A_2$ converges to 0
  as $n\to\infty$ since $\lim_{n\to\infty}b_{ndd,p}=0$.
  
  Furthermore, the term $A_3$ is bounded above by
  \begin{align*}
    n^{-p}
    &d^2\big(\E  |S_{n,1}^{(d)}|^{p}\big)^2
      - n^{-p/2}\sigma^p
      M_pd\E |S_{n,1}^{(d)}|^{p}
      +\sigma^{2p}M_p^2-n^{-p/2}\sigma^p
      M_pd\E |S_{n,1}^{(d)}|^{p}\\
    &=n^{-p/2}d\E |S_{n,1}^{(d)}|^{p}
      \Big(n^{-p/2}d\E |S_{n,1}^{(d)}|^{p}
      -\sigma^p M_p\Big)
      +\sigma^p M_p\Big(\sigma^p M_p-n^{-p/2}
      d\E |S_{n,1}^{(d)}|^{p}\Big)\\
    &=\Big(n^{-p/2}d\E |S_{n,1}^{(d)}|^{p}
      -\sigma^p M_p\Big)^2,
  \end{align*}
  which converges to 0 as $n\to\infty$ by the moment convergence
  theorem.
\end{proof}
\begin{theorem}[Uniform convergence of the $\ell_p$-norm of random 
    walk when $p > 1$]
  \label{thm_p_unif}
  Let $p \in (1 ,\infty)$. Consider a random walk with increments given by
  \eqref{equ_spe_form}. Then
  \begin{align*}
    \sup_{t \in [0,1]} \Big| n^{-p/2}\|S_{\nt}^{(d)}\|_{p}^{p}
    - t^{p/2} \sigma^p M_p \Big|
    \overset{p}{\to} 0 \qquad \text{as} \ n\to\infty.
  \end{align*}
\end{theorem}

\begin{proof}
  For all $p > 1$,
  \begin{displaymath}
    \|S_n^{(d)} \|_p^p =
    \sum_{i = 1}^{d} |S_{n,i}^{(d)}|^p\\
    =T_n^{(d)} + Q_n^{(d)},
  \end{displaymath}
  where
  \begin{equation}
    \label{eq:TQ_p}
    T_n^{(d)}
    = \sum_{i = 1}^{d} |S_{n,i}^{(d)}|^p
    - p \sum_{i = 1}^{d} \sum_{j = 1}^{n} X_{j,i}^{(d)}
    S_{j-1,i}^{(d)}
    |S_{j-1,i}^{(d)}|^{p - 2},\quad
    Q_n^{(d)}
    = p \sum_{i = 1}^{d} \sum_{j = 1}^{n} X_{j,i}^{(d)}
    S_{j-1,i}^{(d)}
    |S_{j-1,i}^{(d)}|^{p - 2}.
  \end{equation}

  For all $n \in \mathbb{N}$,
  \begin{align*}
    T_n^{(d)} - T_{n - 1}^{(d)}
    &= \sum_{i = 1}^{d} \Big(|S_{n,i}^{(d)}|^p
      - p \sum_{j = 1}^{n} X_{j,i}^{(d)} S_{j-1,i}^{(d)}
      |S_{j-1,i}^{(d)} |^{p - 2} 
      - |S_{n-1,i}^{(d)}|^p
      + p \sum_{j = 1}^{n - 1} X_{j,i}^{(d)}  S_{j-1,i}^{(d)}
      |S_{j-1,i}^{(d)} |^{p - 2} \Big) \\
    & = \sum_{i = 1}^{d} \Big( |S_{n,i}^{(d)}|^p
      - |S_{n-1,i}^{(d)}|^p -
      p X_{n,i}^{(d)} S_{n-1,i}^{(d)} 
      | S_{n-1,i}^{(d)} |^{p - 2}\Big)\\
    &=\sum_{i = 1}^{d} \big(|S_{n,i}^{(d)}|^p - |S_{n-1,i}^{(d)}|^p
      - p (S_{n,i}^{(d)} - S_{n-1,i}^{(d)}) S_{n-1,i}^{(d)} |S_{n-1,i}^{(d)}|^{p - 2}\big).
  \end{align*}
  The generic term of this sum is
  \begin{align*}
    |S_{n,i}^{(d)}|^p - |S_{n-1,i}^{(d)}|^p 
    &- p (S_{n,i}^{(d)} - S_{n-1,i}^{(d)}) S_{n-1,i}^{(d)} |S_{n-1,i}^{(d)}|^{p - 2}\\
    &= |S_{n,i}^{(d)}|^p - |S_{n-1,i}^{(d)}|^p - p S_{n,i}^{(d)}  S_{n-1,i}^{(d)} |S_{n-1,i}^{(d)}|^{p - 2} + p |S_{n-1,i}^{(d)}|^{p} \\
    & = |S_{n,i}^{(d)}|^p + (p - 1) |S_{n-1,i}^{(d)}|^{p} - p S_{n,i}^{(d)}  S_{n-1,i}^{(d)} |S_{n-1,i}^{(d)}|^{p - 2} \\ 
    & \geq |S_{n,i}^{(d)}|^p + (p - 1) |S_{n-1,i}^{(d)}|^{p} - p \bigg( \frac{|S_{n,i}^{(d)}|^p}{p}
      +\frac{|S_{n-1,i}^{(d)}|^p}{p / (p - 1)} \bigg) = 0,
  \end{align*}
  where the last inequality follows from
  $xy \leq \frac{x^p}{p} + \frac{y^q}{q}$ with
  $\frac{1}{p} + \frac{1}{q} = 1$, for all $p, q > 1$ and $x, y > 0$.
  Hence, $T_n^{(d)} - T_{n - 1}^{(d)} \geq 0$.
  Thus, the sequence $T_n^{(d)}$ is monotone increasing.

  Next, $Q_n^{(d)}$ is a martingale, since
  \begin{align*}
    \E \Big(Q_n^{(d)} - Q_{n - 1}^{(d)} \ \Big| \ \sF_{n - 1}^{(d)}\Big)
    &= \E \Bigg(\sum_{i = 1}^{d} p X_{n,i}^{(d)}
      S_{n-1,i}^{(d)}
      \big| S_{n-1,i}^{(d)} \big|^{p - 2}
      \ \bigg| \ \sF_{n - 1}^{(d)} \Bigg) \\
    &= p S_{n-1,i}^{(d)} 
      \big| S_{n-1,i}^{(d)} \big|^{p - 2}
      d^{1-1/p} \E \xi= 0,
  \end{align*}
  where $\mathcal{F}_{n - 1}^{(d)}$ is the $\sigma$-algebra generated
  by $X_1^{(d)}, \ldots, X_{n - 1}^{(d)}$. Then, by Doob's inequality,
  \begin{equation}
    \label{equ_Q_diff_unif}
    \mathbf{P} \Big\{ \sup_{t \in [0, 1]} |Q_{\nt}^{(d)}| \geq n^{p /2} \epsilon \Big\}
    \leq n^{-p} \eps^{-2}\E (Q_n^{(d)})^2.
  \end{equation}
  The second moment of $Q_{n}^{(d)}$ is calculated as follows,
  \begin{align*}
    & \E (Q_n^{(d)})^2
      = p^2\sum_{i = 1}^{d}\sum_{j = 1}^{n}\sum_{i' = 1}^{d}\sum_{j' = 1}^{n}
      \E \Big( X_{j,i}^{(d)} X_{j',i'}^{(d)}
      S_{j-1,i}^{(d)}|S_{j-1,i}^{(d)}|^{p - 2}S_{j'-1,i'}^{(d)}
      |S_{j'-1,i'}^{(d)}|^{p - 2} \Big) \\
    & = p^2 \sum_{i = 1}^{d} \sum_{j = 1}^{n} \E (X_{j,i}^{(d)})^2
      \E |S_{j-1,i}^{(d)}|^{2p - 2} \\
    & = p^2 d\E (X_{1,1}^{(d)})^2
      \sum_{j = 1}^{n} (\sqrt{j - 1})^{2p - 2}
      \Big(\E (X_{1,1}^{(d)})^2 \Big)^{p - 1}
      \E \Bigg| \frac{\sum_{l = 1}^{j - 1} X_{l,1}^{(d)}}
      {\sqrt{j - 1} \sqrt{\E (X_{1, 1}^{(d)})^2}} \Bigg|^{2p - 2} \\
    & \leq p^2 d\E (X_{1,1}^{(d)})^2
      \sum_{j = 1}^{N(\delta)} (\sqrt{j - 1})^{2p - 2} \Big(\E(X_{1,1}^{(d)})^2 \Big)^{p - 1}
      \E \Bigg| \frac{\sum_{l = 1}^{j - 1} X_{l,1}^{(d)}}
      {\sqrt{j - 1} \sqrt{\E (X_{1, 1}^{(d)})^2}} \Bigg|^{2p - 2} \\
    & + p^2 d\E (X_{1,1}^{(d)})^2
      \sum_{j = 1}^{n} (\sqrt{j - 1})^{2p - 2}
      \Big(\E (X_{1,1}^{(d)})^2 \Big)^{p - 1} (M_{2p-2}+ \delta),
  \end{align*}
  where the last inequality is implied by Theorem
  \ref{thm:mom_gene}, for all $\delta>0$, there exists an integer $N(\delta)$
  such that for all $j\geq N(\delta)$,
  \begin{displaymath}
    M_{2p-2}-\delta\leq\E \Bigg| \frac{\sum_{l = 1}^{j - 1} X_{l,1}^{(d)}}
    {\sqrt{j - 1} \sqrt{\E (X_{1, 1}^{(d)})^2}} \Bigg|^{2p - 2}\leq M_{2p-2}+\delta.
  \end{displaymath}
  Hence, 
  \begin{align*}
    \E (Q_n^{(d)})^2 &\leq p^2 d
      \E (X_{1,1}^{(d)})^2 \sum_{j = 1}^{N(\delta)} 
      \E \Big| \sum_{l = 1}^{j - 1} X_{l,1}^{(d)}\Big|^{2p - 2} 
      + p^2 (M_{2p-2}+\delta) n^p d
      \big(\E (X_{1,1}^{(d)})^2 \big)^p \\
    & \leq p^2d\E(X_{1,1}^{(d)})^2N(\delta)B_{2p-2}
    \max(N(\delta)^{p-1},N(\delta))\E|X_{1,1}^{(d)}|^{2p-2}\\
    &\hspace{8cm}+ p^2 (M_{2p-2}+\delta) n^p d
      \big(\E (X_{1,1}^{(d)})^2 \big)^p,
  \end{align*}
  where Lemma \ref{lem_Mar} is used to bound the first term and $B_{2p-2}$ 
  is a constant which depends on $p$. The last step is bounded above by
  \begin{displaymath}
    p^2N(\delta)B_{2p-2}\max(N(\delta)^{p-1},N(\delta))d^{-1}
    \E\xi^2\E|\xi|^{2p-2}
    +p^2cn^pd^{-1}\big(\E\xi^2\big)^p.
  \end{displaymath}
   For $p\geq 2$, by Lemma \ref{lem_Mar}, there is an 
   alternative way to bound $\E(Q_n^{(d)})^2$, that is,
   \begin{align*}
     \E(Q_n^{(d)})^2=p^2 \sum_{i = 1}^{d} \sum_{j = 1}^{n} \E (X_{j,i}^{(d)})^2
      \E |S_{j-1,i}^{(d)}|^{2p - 2}
     &\leq p^2dn\E (X_{1,1}^{(d)})^2
     B_{2p-2}n^{p-1}\E|X_{1,1}^{(d)}|^{2p-2}\\
     &=p^2n^pB_{2p-2}d^{-1}\E\xi^2\E|\xi|^{2p-2}.
   \end{align*}
  Then we conclude that
  \begin{displaymath}
    n^{-p} \epsilon^{-2}\E (Q_n^{(d)})^2 \to 0
    \qquad \text{as} \ n\to\infty,
  \end{displaymath}
  and
  \begin{displaymath}
    n^{-p/2}\sup_{t \in [0, 1]}|Q_{\nt}^{(d)}|
    \overset{p} \to 0 \qquad \text{as} \ n\to\infty.
  \end{displaymath}

  By Theorem \ref{thm_general_l_p},
  \begin{displaymath}
    n^{-p/2}T_{\nt}^{(d)}
    = n^{-p / 2}\big(\|S_{\nt}^{(d)} \|_p^p - Q_{\nt}^{(d)}\big)
    \overset{p} \to t^{p / 2} \sigma^p M_p
    \qquad \text{as} \ n\to\infty,
  \end{displaymath}
  for all $t \in [0, 1]$. By monotonicity of the function
  $t \mapsto T_{\nt}^{(d)}$, Dini's theorem yields that
  \begin{displaymath}
    \sup_{t \in [0, 1]} \Big| n^{-p/2}T_{\nt}^{(d)}
    - t^{p / 2} \sigma^p M_p \Big|
    \overset{p} \to 0 \qquad \text{as} \ n\to\infty.
  \end{displaymath}
  Therefore,
  \begin{displaymath}
    \sup_{t \in [0,1]} \Big| n^{-p/2}\|S_{\nt}^{(d)}\|_{p}^{p}
    - t^{p/2} \sigma^p M_p \Big| 
    \leq \sup_{t \in [0, 1]} \Big| n^{-p/2}T_{\nt}^{(d)}
    - t^{p / 2} \sigma^p M_p \Big|
    + n^{-p/2}\sup_{t \in [0, 1]} |Q_{\nt}^{(d)}| \overset{p}\to 0,
  \end{displaymath}
  which completes the proof.
\end{proof}

\begin{theorem}[Uniform convergence for the $\ell_1$-norm of random 
    walk]
  \label{cor_p_unif}
  Consider a random walk with increments given by \eqref{equ_spe_form}.
  Then
  \begin{align*}
    \sup_{t \in [0,1]} \Big| n^{-1/2}\|S_{\nt}^{(d)}\|_1
    - t^{1/2} \sigma M_1 \Big|
    \ \overset{p}{\to} 0 \qquad \text{as} \ n\to\infty.
  \end{align*}
\end{theorem}

\begin{proof}
  It is clear that
  $\|S_n^{(d)} \|_{1}$ can be expressed as follows,
  \begin{displaymath}
    \|S_n^{(d)} \|_1
    = \sum_{i = 1}^{d} |S_{n,i}^{(d)}|
    = T_n^{(d)} + Q_n^{(d)},
  \end{displaymath}
  where
  \begin{equation}
    \label{eq:TQ_p_1}
    T_n^{(d)}
    = \sum_{i = 1}^{d} \big|S_{n,i}^{(d)}\big|
    - \sum_{i = 1}^{d} \sum_{j = 1}^{n} X_{j,i}^{(d)}
    S_{j-1,i}^{(d)}
    |S_{j-1,i}^{(d)}|^{-1} , \quad
    Q_n^{(d)}
    = \sum_{i = 1}^{d} \sum_{j = 1}^{n} X_{j,i}^{(d)}
    S_{j-1,i}^{(d)}|S_{j-1,i}^{(d)}|^{-1}.
  \end{equation}
  The sequence $T_n^{(d)}$ is monotone increasing, since
  \begin{align*}
    T_n^{(d)} - T_{n - 1}^{(d)} 
    &= \sum_{i = 1}^{d} \Bigl( |S_{n,i}^{(d)}|
    - |S_{n-1,i}^{(d)}| - X_{n,i}^{(d)}
    S_{n-1,i}^{(d)} \big|
    S_{n-1,i}^{(d)} \big|^{-1} \Bigl)\\
    & = \begin{cases}
      |S_{n,i}^{(d)}| - S_{n,i}^{(d)}, & \text{if $ S_{n-1,i}^{(d)} > 0$},\\
      |S_{n,i}^{(d)}| + S_{n,i}^{(d)}, & \text{if $ S_{n-1,i}^{(d)} < 0$},
    \end{cases}
                     \ \geq 0.
  \end{align*}
   
  Next, $Q_n^{(d)}$ is a martingale, since
  \begin{align*}
    \E \Big(Q_n^{(d)} - Q_{n - 1}^{(d)} \ \Big|\ \sF_{n - 1} \Big)
    &= \E \Big(\sum_{i = 1}^{d} X_{n,i}^{(d)}
      S_{n-1,i}^{(d)} 
      \big| S_{n-1,i}^{(d)} \big|^{-1} \ \Big|\ \sF_{n - 1} \Big) \\
    &= S_{n-1,i}^{(d)} 
      \big| S_{n-1,i}^{(d)} \big|^{-1}
      d^{1-1/p}\E \xi= 0,
  \end{align*}
  where $\mathcal{F}_{n - 1}^{(d)}$ is the $\sigma$-algebra
  generated by $X_1^{(d)}, \ldots, X_{n - 1}^{(d)}$.

  Then, by Doob's inequality,
  \begin{equation}
    \label{equ_p_1_1}
    \mathbf{P} \Big\{ \sup_{t \in [0, 1]} |Q_{\nt}^{(d)}| \geq n^{1 /2} \eps \Big\}
    \leq n^{-1} \epsilon^{-2}\E (Q_n^{(d)})^2
  \end{equation}
  The second moment of $Q_{n}^{(d)}$ is calculated as follows,
  \begin{align*}
    & \E (Q_n^{(d)})^2
      = \E \Big( \sum_{i = 1}^{d} \sum_{j = 1}^{n} X_{j,i}^{(d)}
      S_{j-1,i}^{(d)} |S_{j-1,i}^{(d)} |^{-1}\Big)^2 \\
    & = \sum_{i = 1}^{d} \sum_{j = 1}^{n} \sum_{i' = 1}^{d} \sum_{j' =1}^{n}
      \E \Big( X_{j,i}^{(d)} X_{j',i'}^{(d)}
      S_{j-1,i}^{(d)} |S_{j-1,i}^{(d)} |^{-1}
      S_{j'-1,i'}^{(d)}|S_{j'-1,i'}^{(d)}|^{-1} \Big) \\
    & = \sum_{i = 1}^{d} \sum_{j = 1}^{n} \E (X_{j,i}^{(d)})^2
      = n d\E (X_{1,1}^{(d)})^2
      = n  d^{-1}\E\xi^2.
  \end{align*}
  Then
  \begin{displaymath}
    n^{-1}\eps^{-2}\E (Q_n^{(d)})^2
    =\eps^{-2}d^{-1}\E\xi^2\to 0 \qquad \text{as} \ d\to\infty,
  \end{displaymath}
  and 
  \begin{displaymath}
    n^{-1/2}\sup_{t \in [0, 1]}|Q_{\lfloor nt \rfloor}^{(d)}|
    \overset{p} \to 0  \qquad \text{as} \ n\to\infty.
  \end{displaymath}
  
  Finally, the proof is completed by following the similar strategy as 
  in Theorem \ref{thm_p_unif}.
\end{proof}

\begin{theorem}[Uniform convergence of the $\ell_p$-metric
  of differences]
\label{equ_unif_lp}
Let $p \in [1, \infty)$. Consider a random walk with increments given by
\eqref{equ_spe_form}. Then
\begin{align*}
  \sup_{0 \leq s \leq t \leq 1} \Big|n^{-1/2}\|S_{\nt}^{(d)} -S_{\ns}^{(d)} \|_{p}
  - \sqrt{t - s} \sigma  M_p^{1/p} \Big|
  \ \overset{p}{\to} 0 \qquad \text{as} \ n\to\infty.
\end{align*}
\end{theorem}

\begin{proof}
  Take some $m \in \mathbb{N}$. By Theorem \ref{thm_general_l_p},
  for every $i = 0, \ldots, m - 1$,
  \begin{displaymath}
    n^{-1/2}\|S_{\lfloor n(i / m) \rfloor}^{(d)} \|_p
    \overset{p} \to \sqrt{i/m} \sigma  M_p^{1/p} .
  \end{displaymath}
  Moreover, for every integer $0 \leq i \leq j \leq m$,
  by stationarity of the increments of random walks, we obtain that
  \begin{displaymath}
    n^{-1/2}\|S_{\lfloor n(j / m) \rfloor}^{(d)} - S_{\lfloor n(i / m)\rfloor}^{(d)} \|_p
    \overset{p} \to \sqrt{(j-i)/m} \sigma  M_p^{1/p} .
  \end{displaymath}
  By the union bound, it follows that, for every fixed $m \in \mathbb{N}$,
  \begin{displaymath}
    \max_{0 \leq i \leq j \leq m}
    \Big| n^{-1/2}\|S_{\lfloor n(j / m) \rfloor}^{(d)} - S_{\lfloor n(i / m) \rfloor}^{(d)} \|_p
    - \sqrt{(j-i)/m} \sigma  M_p^{1/p}  \Big| \overset{p} \to 0.
  \end{displaymath}
  If $0 \leq s \leq t \leq 1$ are such that
  $s \in [i/m, (i+1)/m)$ and
  $t \in [j/m, (j+1)/m)$, then by the triangle inequality,
  \begin{multline*}
    \Big| n^{-1/2}\| S_{\lfloor nt \rfloor}^{(d)} - S_{\lfloor ns\rfloor}^{(d)} \|_{p}
    -  n^{-1/2}\|S_{\lfloor n(j / m) \rfloor}^{(d)} - S_{\lfloor n(i /m) \rfloor}^{(d)} \|_p\Big| \\
    \leq n^{-1/2}\sup_{z \in [\frac{i}{m}, \frac{i + 1}{m}]}
    \|S_{\lfloor nz \rfloor}^{(d)} - S_{\lfloor n(i / m)\rfloor}^{(d)} \|_p
    + n^{-1/2}\sup_{z \in [\frac{j}{m}, \frac{j + 1}{m}]}
    \| S_{\lfloor nz \rfloor}^{(d)} - S_{\lfloor n(j / m) \rfloor}^{(d)} \|_p.
  \end{multline*}
  Consider the random variable
  \begin{displaymath}
    \eps_{m, d} = n^{-1/2}\max_{i \in \{ 0, \ldots, m - 1 \}}
    \sup_{z \in [\frac{i}{m}, \frac{i + 1}{m}]}
    \| S_{\lfloor nz \rfloor}^{(d)} - S_{\lfloor n(i / m) \rfloor}^{(d)} \|_p.
  \end{displaymath}
  To complete the proof, it suffices to show that, for every $\epsilon > 0$,
  \begin{displaymath}
    \lim_{m \to \infty} \limsup_{n\to\infty}
    \mathbf{P}\{\epsilon_{m, d} \geq \eps\} = 0.
  \end{displaymath}
  By the union bound, it follows that, for every fixed $m \in \mathbb{N}$,
  \begin{displaymath}
    \mathbf{P}\{\eps_{m ,d} \geq \eps\}
    \leq m \mathbf{P}\bigg\{ \sup_{t \in [0, \frac{1}{m}]}
    \| S_{\nt}^{(d)} \|_p^p \geq n^{p / 2}\eps^p\bigg\}.
  \end{displaymath}
  Since $\|S_{\nt}^{(d)} \|_p^p = T_{\nt}^{(d)} + Q_{\nt}^{(d)}$ with 
  $T_{\nt}^{(d)}$ and $Q_{\nt}^{(d)}$ defined by \eqref{eq:TQ_p} for $p>1$
  and by \eqref{eq:TQ_p_1} for $p=1$, it suffices to show that
  \begin{equation}
    \label{equ_uni_p_1}
    \lim_{m \to \infty} \limsup_{n\to\infty}
    m \mathbf{P} \{T_{\lfloor n / m \rfloor}^{(d)} \geq n^{p / 2} \eps^p / 2\} = 0,
  \end{equation}
  and
  \begin{equation}
    \label{equ_uni_p_2}
    \lim_{m \to \infty} \limsup_{n\to\infty}
    m \mathbf{P} \{\sup_{t \in [0, \frac{1}{m}]} Q_{\nt}^{(d)} \geq n^{p / 2} \eps^p / 2\} = 0.
  \end{equation}
  For fixed $m \in \mathbb{N}$,
  \begin{displaymath}
    n^{-p/2}T_{\lfloor n / m \rfloor}^{(d)}
    \overset{p} \to m^{-p/2} \sigma^p M_p\quad\text{as}\ n\to\infty.
  \end{displaymath}
  Hence, \eqref{equ_uni_p_1} holds for every
  $m > 2^{2/p} \sigma^2 (M_p)^{2 / p} /\eps^2$. By Doob's inequality,
  \begin{displaymath}
    m \mathbf{P} \{\sup_{t \in [0, \frac{1}{m}]}
    Q_{\nt}^{(d)} \geq n^{p / 2} \eps^p / 2\}
    \leq m n^{-p}(\eps^p / 2)^{-2}\E (Q_{\lfloor n / m \rfloor}^{(d)})^2\to 0,
  \end{displaymath}
  where the last step is implied by \eqref{equ_Q_diff_unif} for $p > 1$ and by \eqref{equ_p_1_1}
  for $p = 1$, hence \eqref{equ_uni_p_2} holds.
\end{proof}

\begin{proof}[Proof of Theorem \ref{the_pair_lp}.]
  By Corollary 7.3.28 of \cite{MR1835418}, the Gromov-Hausdorff distance 
  between $(n^{-1/2}\sZ_{n},\|\cdot\|_p)$ and 
  $\big([0,1], \sqrt{|t-s|} \sigma
  M_p^{1/p} \big)$ is bounded by
  \begin{displaymath}
    2 \sup_{0 \leq s \leq t \leq 1} \Big| 
    n^{-1/2}\| S_{\nt}^{(d)}-S_{\ns}^{(d)} \|_{p}
    -\sqrt{t - s} \sigma  M_p^{1/p}  \Big|.
  \end{displaymath}
  The proof is completed by referring to Theorem \ref{equ_unif_lp}.
\end{proof}

\section*{Acknowledgements}
The author is grateful to Ilya Molchanov for his suggestions and
patience in instructing and correcting this paper.

\bibliographystyle{unsrt}
\bibliography{paper_1}

\end{document}